\documentclass[11pt,hyp,]{nyjm}

\usepackage{hyperref}
\hypersetup{nesting=true,debug=true,naturalnames=true}
\usepackage{graphicx,amssymb,upref}

\let\<\langle
\let\>\rangle
\usepackage[all,pdf]{xy}
\UseComputerModernTips

\let\uml\"

\title[An elliptic sequence is not a sampled
linear sequence]{An elliptic sequence is not a sampled linear recurrence sequence}

\author{F.~Luca}
\address{School of Mathematics\\
University of the Witwatersrand\\
Private Bag 3, Wits 2050, South Africa}
\email{Florian.Luca@wits.ac.za}
\author{T.~Ward}
\address{Ziff Building\\
University of Leeds\\
Leeds LS2 9JT, UK}
\email{t.b.ward@leeds.ac.uk}



\renewcommand\ge{\geqslant}
\renewcommand\le{\leqslant}
\def\seqn#1{(#1)_{n\ge1}}
\def\seq#1{(#1)}

\newtheorem{theorem}{Theorem}
\newtheorem{lemma}[theorem]{Lemma}

\def\eul{{\rm{e}}}
\def\imag{{\rm{i}}}
\def\C{\mathbb C}
\def\F{\mathbb F}
\def\L{\mathbb L}
\def\K{\mathbb K}
\def\Q{\mathbb Q}
\def\Z{\mathbb Z}

\begin{document}

\begin{abstract}
Let~$E$ be an elliptic curve defined over the rationals
and in minimal Weierstrass form, and let~$P=(x_1/z_1^2,y_1/z_1^3)$
be a rational point of infinite order on~$E$, where~$x_1,y_1,z_1$
are coprime integers. We show that the integer
sequence~$\seqn{z_n}$ defined by~$nP=(x_n/z_n^2,y_n/z_n^3)$ for all~$n\ge 1$
does not eventually coincide with~$\seqn{u_{n^2}}$
for any choice of linear recurrence sequence~$\seqn{u_n}$ with integer values.
\end{abstract}
\maketitle
\tableofcontents

\section{Introduction}

Let~$E$ be an elliptic curve defined over~${\mathbb Q}$,
given by an equation of the form
\begin{equation}\label{eq:1}
y^2=x^3+Ax+B
\end{equation}
with~$A,B\in\Z$ and
with discriminant~$\Delta_E=4A^3+27B^2\ne 0$,
and write all affine points in the form~$(x/z^2,y/z^3)$ with~$\gcd(x,y,z)=1$ and~$z>0$.
Let~$P=(x_1/z_1^2,y_1/z_1^3)\in E(\mathbb{Q})$
have infinite order
and associate to~$P$ the integer sequence~$\seqn{z_n}$
where~$nP=(x_n/z_n^2,y_n/z_n^3)$ for all~$n\ge 1$.
We will refer to such sequences as being
elliptic divisibility sequences
(there are several different definitions and we
will only be cavalier about the distinction where it
does not matter).
It is known that such a sequence
has a characteristic quadratic-exponential growth rate,~$\log z_n=(c+o(1))n^2$
as~$n\to\infty$
(see~\cite[Sec.~10.4]{MR1990179} for a discussion of the
relation between sequences of this form and elliptic
divisibility sequences defined via a bilinear recurrence or
the sequence of division polynomials of the curve,
and for references to some of the basic
facts about linear and bilinear recurrence sequences including
the growth rate). The constant~$c$ is the canonical height of the
point~$P$ on the curve~$E$.

On the other hand, an integer sequence~$\seqn{u_n}$
is said to be a linear recurrence sequence
of order~$k\ge1$ if there are constants~$c_1,\dots,c_k$
with~$c_k$ non-zero satisfying
\begin{equation}\label{equation:LRSoforderk}
u_{n+k}=c_1u_{n+k-1}+\cdots+c_ku_n
\end{equation}
for all~$n\ge1$, and~$k$ is minimal with this property.
By Fatou's lemma~\cite[p.~369]{MR1555035}
we may assume that~$c_1,\ldots,c_k$ are also integers.
It is known that such a sequence (under a non-degeneracy
hypothesis detailed later) has a characteristic linear-exponential
growth rate: for any~$\epsilon>0$ there is some~$N=N(\epsilon,\seq{u_n})$
and constants~$A,C>0$ with~$C^{(1-\epsilon)n}\le\vert u_n\vert
\le AC^nn^k$ for all~$n\ge N$ (see Evertse~\cite{MR766298}
or van der Poorten and Schlickewei~\cite{MR93d:11036}).
The deep part of this statement
is to control possible cancellation between dominant
characteristic roots of equal size. We will not use this
result here, but instead will deal directly
with the possible multiplicity of dominant roots.
Here the characteristic
growth parameter~$C$ is the maximum of the set of absolute
values of zeros of the associated characteristic polynomial.

It also makes sense to ask
questions about arithmetic properties. For example:
\begin{itemize}
\item Does the sequence have a `Zsigmondy bound', meaning that eventually
each term of the sequence has a prime divisor that does
not divide any earlier term? Silverman~\cite[Lemma~9]{MR961918} has shown that
an elliptic divisibility sequence always has this property,
and this will be used below.
Some linear recurrence sequences do have this property
(including, in particular, all Lucas and Lehmer sequences)
and some do not.
\item Does the sequence count periodic points for some
map? Here it is known that some -- but far from all -- linear
recurrence sequences do, while Silverman and Stephens~\cite{MR2226354}
show that no elliptic divisibility sequence does.
\end{itemize}

In light of the growth rate observations particularly, it is natural to ask if
an elliptic divisibility sequence is simply a linear
recurrence sequence in disguise, obtained by sampling
the linear recurrence sequence at the squares. Our purpose here
is to show that this is not the case in the following
robust sense. Let us say that sequences~$\seqn{a_n}$
and~$\seqn{b_n}$ are \emph{eventually equal}, written~$\seqn{a_n}=_e\seqn{b_n}$,
if there is some~$N=N(\seq{a_n},\seq{b_n})$ with~$a_n=b_n$ for all~$n\ge N$.

\begin{theorem}\label{thm:1}
Let~$E:y^2=x^3+Ax+B$ be an elliptic
curve defined over the rationals,
let~$P=(x_1/z_1^2,y_1/z_1^3)\in E(\Q)$ be a point
of infinite order, and let~$\seqn{z_n}$
be a sequence of integers satisfying~$nP=(x_n/z_n^2,y_n/z_n^3)$ (the sign of $z_n$ can be chosen arbitrarily).
Then no integer linear recurrence sequence~$\seqn{u_n}$ has the property
that
\[
\seqn{z_{n^{\vphantom{2}}}}=_e\seqn{u_{n^2}}.
\]
\end{theorem}

For any such sequence ~$\seqn{z_n}$ there is some~$\ell\ge1$
with the property that~$\seqn{z_{\ell n}}$ is, up to signs, an
elliptic divisibility sequence in the recurrence sense, meaning that it
satisfies the
non-linear recurrence
defined by specifying
four initial integral values~$w_1,w_2,w_3,w_4$ with~$w_1w_2w_3\neq0$
and with~$w_2\vert w_4$,
and satisfies
\begin{equation}\label{edsrecurrence1}
w_{2n+1}^{\vphantom{3}}w_1^3=w_{n+2}^{\vphantom{3}}w_n^3-w_{n+1}^3w_{n-1}^{\vphantom{3}}
\end{equation}
for~$n\ge2$ and
\begin{equation}\label{edsrecurrence2}
w_{2n}^{\vphantom{3}}w_2^{\vphantom{3}}w_1^2=w_{n+2}^{\vphantom{3}}w_n^{\vphantom{3}}w_{n-1}^2
-
w_n^{\vphantom{3}}w_{n-2}^{\vphantom{3}}w_{n+1}^2.
\end{equation}
Since the property of being a linear recurrence sequence is
preserved under the operation taking~$\seqn{u_n}$ to~$\seqn{u_{\ell^2n}}$,
it is therefore enough to show that Theorem~\ref{thm:1}
holds for elliptic divisibility sequences defined either
geometrically using a non-torsion point on an elliptic
curve or using the non-linear recurrence relation.

For a restricted class of linear recurrence sequences
we can already deduce Theorem~\ref{thm:1} from the work
of Silverman and Stephens~\cite{MR2226354},
by the following argument. Moss~\cite[Th.~2.2.2]{moss} has
given a combinatorial proof that
if~$\seqn{u_n}$ counts the periodic points for some
map, then so does~$\seqn{u_{n^k}}$ for any~$k\ge1$.
It follows that (for example) sampling a Lehmer--Pierce
sequence along the squares never produces an elliptic
divisibility sequence.

We shall give two proofs of Theorem~\ref{thm:1}, a
complex (Diophantine) one, which works only when the signs of $z_n$ are chosen in a specific way, and a~$p$-adic
(arithmetic) one which works for any choice of signs.

\section{A Diophantine proof of a special case of the main theorem}

This particular proof works when $(z_n)_{n\ge 1}$ is an elliptic divisibility sequence. We make this assumption throughout this section.  
We start with a linear recurrence sequence~$\seqn{u_n}$ of some order~$k\ge1$,
assume the relation
\begin{equation}\label{eqn:eventualequivalence}
\seqn{z_{n^{\vphantom{2}}}}=_e\seqn{u_{n^2}},
\end{equation}
deduce certain
properties the linear recurrence sequence must have, and finally
argue that the hypothesis leads to a contradiction. If~$k=1$
then~$\seqn{u_n}$ is either constant or a geometric progression
and so in particular the largest prime factor of~$u_n$ is bounded. On the
other hand, as mentioned above, Silverman~\cite[Lemma~9]{MR961918}
has
shown that all but finitely many terms of~$\seqn{z_n}$
have a primitive prime divisor (that~$E$ is really an elliptic
curve -- it has non-vanishing discriminant -- is used here),
and so the largest prime
divisor of~$\seqn{z_n}$, and hence of~$\seqn{u_n}$, cannot be bounded.
It follows that~$k>1$.

Assume therefore that~$\seqn{u_n}$ has order~$k\ge2$ and satisfies~\eqref{equation:LRSoforderk};
write
\[
\Psi(x)=x^k-c_1x^{k-1}-\cdots-c_k=\prod_{i=1}^{s} (x-\alpha_i)^{\sigma_i}
\]
where~$\alpha_1,\ldots,\alpha_s\in\C$ are distinct roots
with multiplicity~$\sigma_1,\ldots,\sigma_s$ respectively.
As usual we may then write the terms of the sequence as
a generalized power sum
\begin{equation}
\label{eq:Binet}
u_n=\sum_{i=1}^r P_i(n) \alpha_i^n,
\end{equation}
for all~$n\ge1$, where the polynomials~$P_i(X)\in\Q(\alpha_1,\ldots,\alpha_s)[X]$
have degree~$(\sigma_i-1)$ for~$i=1,\ldots,s$
(the claim on the degrees being a consequence of the assumed minimality
of~$k$; taking the form of~\eqref{eq:Binet} in fact
characterizes being a linear recurrence sequence of order no more than~$k$,
which implies useful consequences like~$\seqn{u_{mn}}$ being a linear
recurrence of order no more than~$k$ for any~$m\ge1$ if~$\seqn{u_n}$ is
a linear recurrence of order~$k$, for instance).

We next claim that -- for the purposes of proving Theorem~\ref{thm:1} --
we may assume that~$\seqn{u_n}$ is non-degenerate. This is a standard
reduction argument in the study of linear recurrence sequences,
which we outline briefly. A linear
recurrence sequence of order~$k$ written as~\eqref{eq:Binet}
is said to be degenerate if for some pair~$1\le i\neq j\le s$
the quotient~$\alpha_i/\alpha_j$ is a root of unity, and non-degenerate
if not. Since the group of roots of unity in~$\K=\Q(\alpha_1,\dots,\alpha_s)$
is a finite cyclic group, there is some~$M$ with the property
that if a product~$\zeta=\alpha_1^{m_1}\cdots\alpha_s^{m_s}$
is a root of unity, then~$\zeta^M=1$.
Thus we may replace the sequence~$\seqn{u_n}$ with~$\seqn{u_{M^2n}}$,
which is clearly a linear recurrence sequence of order no more than~$k$ by~\eqref{eq:Binet},
and the relation~\eqref{eqn:eventualequivalence}
implies that~$\seqn{z_{Mn}}=_e\seqn{u_{M^2n}}$,
which is the same relation but with the point~$P$ replaced with~$MP$.
Here we are taking advantage of the geometric description of the
elliptic divisibility sequence.
Thus it is sufficient to show Theorem~\ref{thm:1} for non-degenerate
linear recurrence sequences of order~$k\ge2$.
By rescaling once again (which will not affect the
non-degeneracy), we
may also assume that the elliptic divisibility sequence satisfies the non-linear
recurrence~\eqref{edsrecurrence1}--\eqref{edsrecurrence2}.

Re-label the
zeroes of~$\Psi$ so that
\[
\vert\alpha_j\vert=\rho=\max\{|\alpha_i|\mid1\le i\le s\}>1
\]
for~$j=1,\dots,r$
and~$\vert\alpha_j\vert\le \rho^{1-\delta}$ for~$j=r+1,\ldots,s$ for some~$\delta>0$ (that~$\rho>1$ follows from~\eqref{eqn:eventualequivalence} and the fact that the sequence~$\seqn{z_n}$
grows like~$c^{n^2}$ for some~$c>1$,
since the canonical height of a non-torsion point
is positive).
So we may write~$\alpha_j=\rho{\eul}^{{\imag}\theta_j}$ with~$\theta_j\in (-\pi,\pi]$
for~$j=1,\ldots,r$, and
\[
u_n
=
\sum_{i=1}^r P_i(n)\alpha_i^{n}+O\left(n^D \rho^{n(1-\delta)}\right),
\]
where~$D:=\max\{\sigma_i: 1\le i\le s\}.$

From~\eqref{edsrecurrence1} and~\eqref{edsrecurrence2}
we have
\[
z_{2n+1}^{\vphantom{3}}=z_{n+2}^{\vphantom{3}}z_{n}^3-z_{n-1}^{\vphantom{3}}z_{n+1}^3
\]
for all~$n\ge 1$. Using~\eqref{eqn:eventualequivalence}, we
deduce that
\begin{align}
\sum_{i=1}^{r}  P_i((2n+1)^2) \alpha_i^{(2n+1)^2}
&=
\left(\vphantom{\sum}\right.\!\!\sum_{i=1}^{r} P_i((n+2)^2) \alpha_i^{(n+2)^2}\!\!\left.\vphantom{\sum}\right)
\negmedspace\times\negmedspace
\left(\vphantom{\sum}\right.\!\!\sum_{i=1}^r P_i(n^2)\alpha_i^{n^2}\!\!\left.\vphantom{\sum}\right)^3\nonumber\\
&\hspace{-2cm}-
\left(\vphantom{\sum}\right.\!\!\sum_{i=1}^r P_i((n-1)^2)\alpha_i^{(n-1)^2}\!\!\left.\vphantom{\sum}\right)
\negmedspace\times\negmedspace
\left(\vphantom{\sum}\right.\!\!\sum_{i=1}^r P_i((n+1)^2) \alpha_i^{(n+1)^2}\!\!\left.\vphantom{\sum}\right)^3\nonumber\\
&+
O\!\!\left(\!n^{8D} \rho^{4n^2+4n-\delta n^2}\!\right)\label{eq:4}
\end{align}
for large~$n\ge1$. So, it makes sense to consider the expression
\begin{align*}
F(X,Z_1,\ldots,Z_r)
&:=
\sum_{i=1}^{r}  P_i((2X+1)^2) Z_i^{(2X+1)^2}\\
&\hspace{-1cm}-
\left(\sum_{i=1}^{r} P_i((X+2)^2) Z_i^{(X+2)^2}\right)
\left(\sum_{i=1}^r P_i(X^2)Z_i^{X^2}\right)^3\nonumber\\
&\hspace{-1cm}+\left(\sum_{i=1}^r P_i((X-1)^2)Z_i^{(X-1)^2}\right)
\left(\sum_{i=1}^r P_i((X+1)^2) Z_i^{(X+1)^2}\right)^3\nonumber\\
&\hspace{-1cm}=:\sum_{j=1}^{L} Q_j(X) M_j(X,Z_1,\ldots,Z_r),
\end{align*}
where the~$Q_j(X)$ are polynomials in the variable~$X$
of degree at most~$8D$, and for fixed positive integer~$X$
the expressions~$M_j(X,Z_1,\ldots,Z_r)$
are monomials in~$Z_1,\ldots,Z_r$ of
degree~$4X^2+4X+1$ or~$4X^2+4X+4$. Here,~$L=r+2r^4$.
Further, up to relabeling of the indices~$j\in \{1,\ldots,L\}$,
we may assume that
\[
(Q_i(X),M_i(X))  =  (P_i((2X+1)^2), Z_i^{4X^2+4X+1})
\]
for~$1\le i\le r$,
and that
\begin{align*}
(Q_i(X),M_i(X))
&=
\Bigl(P_i((X+2)^2) P_i^3(X^2)\\
&\qquad\qquad-
P_i((X-1)^2)P_i^3((X+1)^2), Z_i^{4X^2+4X+4}\Bigr)
\end{align*}
for~$r+1\le j\le 2r$. Note that~$M_j(X)$ involves at least two of the
indeterminates~$Y_1,\ldots,Y_r$ for~$j>2r$ and that~$M_j(X)$
has total degree~$4X^2+4X+4$ for all~$j>r$.
Since we need to specialize the expression~$F(X,Z_1,\ldots,Z_r)$ to~$X=n$
and~$(Z_1,\ldots,Z_r)=(\alpha_1,\ldots,\alpha_r)$,
where the components of this last~$r$-dimensional vector are
multiplicatively independent complex numbers with
the same absolute value~$\rho$,
we find it convenient to make the change of
variable~$Z_i:=Z\eul^{{\imag} Y_i}$,
and thus look at the expression
\begin{align*}
G(X,Z,Y_1,\ldots, Y_r)
&:=
F(X,Z\eul^{{\imag} Y_1},\ldots,Z\eul^{{\imag}Y_r})\\
&\hspace{-2.5cm}:=
\sum_{i=1}^{r}  P_i((2X+1)^2) (Z\eul^{{\imag} Y_i})^{(2X+1)^2}\\
&\hspace{-2.5cm}-
\left(\!\sum_{i=1}^{r} P_i((X+2)^2) (Z\eul^{{\imag} Y_i})^{(X+2)^2}\!\right)\!\!
\left(\!\sum_{i=1}^r P_i(X^2)(Z\eul^{{\imag} Y_i})^{X^2}\!\right)^{\!3}\nonumber\\
&\hspace{-2.5cm}+\left(\sum_{i=1}^r P_i((X-1)^2)(Z\eul^{{\imag} Y_i})^{(X-1)^2}\!\right)\!\!
\left(\!\sum_{i=1}^r P_i((X+1)^2) (Z\eul^{{\imag} Y_i})^{(X+1)^2}\!\right)^{\!3}\nonumber\\
&\hspace{-2.5cm}=Z^{4X^2+4X+1} \sum_{j=1}^{L_1} Q_j(X,Z)_{Y_1,\ldots,Y_r}\eul^{f_j(Y_1,\ldots,Y_r,X)},
\end{align*}
where~$L_1:=L-r$,
\[
Q_j(X,Z)_{Y_1,\ldots,Y_r}
=
\begin{cases}
(Q_j(X)-(Z\eul^{{\imag}Y_j})^3 Q_{j+r}(X))\eul^{{\imag}Y_j}, & 1\le j\le r;\\
Z^3Q_{j+r}(X) M_{j+r}(0,\eul^{{\imag}Y_1},\ldots,\eul^{{\imag}Y_r}), & r+1\le j\le L_1,
\end{cases}
\]
and
\[
\eul^{f_j(Y_1,\ldots,Y_r,X)}
=
\begin{cases}
\eul^{{\imag}Y_j(4X^2+4X)}, & 1\le j\le r;\\
\frac{M_{j+r}(X,\eul^{{\imag} Y_1},\ldots,\eul^{{\imag}Y_r})}{M_{j+r}(0,\eul^{{\imag} Y_1},\ldots,\eul^{{\imag}Y_r})}, & r+1\le j\le L_1.
\end{cases}
\]
Note that the expressions~$f_j(Y_1,\ldots,Y_r,X)$
are linear forms in~${\imag}Y_1,\ldots, {\imag}Y_r$ whose coefficients are
quadratic polynomials in~$X$. In fact,
\[
f_i(Y_1,\ldots,Y_r,X)={\imag} \sum_{j\in I_i} m_{i,j}(X) Y_j,
\]
where~$I_i\in \{1,\ldots,r\}$, ~$m_{i,j}(X)$ are quadratic polynomials in~$X$ with integer coefficients,~$m_{i,j}(0)=0$ for all~$1\le i\le L_1$ and~$j\in I_i$, and
\[
\sum_{j\in I_i} m_{i.j}(X)=4X^2+4
\]
for all~$1\le i\le L_1$.
For~$i=1,\ldots,r$, we have~$I_i=\{i\}$,
therefore
\[
m_{i,i}(X)=4X^2+4X,
\]
while for~$i>r$,~$I_i$ has at least two (and at most four) elements.
Now that we have fixed some notation, we return to~\eqref{eq:4}, put the dominant terms
on the left-hand side, the expression inside~$O$
on the right-hand side, and divide both sides by~$\rho^{4n^2+4n+1}$
obtaining (in our notation)
\begin{equation}\label{eq:t}
\rho^{-4n^2-4n-1} G(n,\rho,\theta_1,\ldots,\theta_r)=\sum_{i=1}^{L_1} x_i=O(\rho^{-\delta_1 n^2}),
\end{equation}
where~$\delta_1:=\delta/2$ and
\[
x_i=x_i(n)=Q_i(n,\rho)_{\theta_1,\ldots,\theta_r} \eul^{{\imag} f_i(n)}
\]
for~$1\le i\le L_1$, with
\[
f_i(n):=f_i(\theta_1,\ldots,\theta_r,n)
\]
for all~$i\in\{1,\ldots,L_1\}$.
Let us take a closer look at
\[
f_i(X)=\sum_{j\in I_i} m_j(X)\theta_j\in {\mathbb C}[x]
\]
for~$i\in \{1,\ldots,L\}$.
We claim that if two elements
in~$\{f_1,\dots,f_L\}$ are equivalent modulo the
equivalence relation
\[
f_{\ell_1}(X)\equiv_\pi f_{\ell_2}(X)
\Longleftrightarrow
\textstyle\frac{1}{\pi} (f_{\ell_1}(X)-f_{\ell_2}(X))\in {\mathbb Q}[X]
\]
then they are in fact equal. To see this,
notice that~$f_{\ell_1}(X)\equiv_\pi f_{\ell_2}(X)$ implies
that
\[
\eul^{{\imag}(f_{\ell_1}(n)-f_{\ell_2}(n))}
\]
is a monomial
in~$\alpha_1/\rho,\dots,\alpha_r/\rho$ and is a root of unity. In particular, for some positive integer~$A$ we have
\[
\eul^{A {\imag}( f_{\ell_1}(n)-f_{\ell_2}(n))}=1.
\]
This leads to
\[
\Bigl(\prod_{j\in I_{\ell_1}} \Bigl(\frac{\alpha_j}{\rho}\Bigr)^{ A m_{\ell_1,j}(n)}\Bigr)
\Bigl(\prod_{j\in I_{\ell_2}} \frac{\alpha_j}{\rho}\Bigr)^{-A m_{\ell_2,j}(n)}=1.
\]
Since~$\sum_{j\in I_{\ell}} m_{\ell,j}(n)$ is equal to~$4n^2+4n$,
this gives
\[
\prod_{j\in I_{\ell_1}} \alpha^{A m_{\ell_1,j}(n)} \prod_{j\in I_{\ell_2}} \alpha_j^{-A m_{\ell_2,j}(n)}=1,
\]
so~$I_{\ell_1}=I_{\ell_2}$ and~$m_{\ell_1,j}(n)=m_{\ell_2,j}(n)$ for~$j\in I_{\ell_1}$
and for all~$n$
since~$\alpha_1,\ldots,\alpha_r$ are multiplicatively independent,
and hence~$f_{\ell_1}(n)=f_{\ell_2}(n)$ for all~$n$.
It follows that~$f_{\ell_1}(X)=f_{\ell_2}(X)$.
In particular, we deduce that~$\eul^{{\imag}(f_{\ell_1}(n)-f_{\ell_2}(n))}$ is not a root of unity
for large~$n$
if~$f_{\ell_1}(X)\ne f_{\ell_2}(X)$.

The method of proof consists now in completing the following three steps:
\begin{itemize}
\item[(a)] For each~$i\in \{1,\ldots,L_1\}$ there is
some~$j\in\{1,\ldots,L_1\}$,~$j\ne i$, such that~$f_{j}=f_{i}$.
\item[(b)] If~$i\in \{1,\ldots,r\}$ and~$f_j=f_j$ for some~$j\ne i$, then~$j\ge r+1$.
\item[(c)] The final contradiction.
\end{itemize}
Let us look at the left-hand side of~\eqref{eq:t}. Assume first that
it is not identically zero as a function of~$n$.
Then the expression on the right-hand side of~\eqref{eq:t} is not
identically zero either, so~$L\ge 1$.
Moreover, the vector
\begin{equation}
\label{eq:vector}
\mathbf{x}(n)=(x_1(n),\ldots,x_L(n))
\end{equation}
satisfies
\begin{equation}
\label{eq:height}
H({\bf x})\ge \rho^{\kappa n^2}
\end{equation}
for some appropriate positive constant~$\kappa$, where~$H$ denotes
the na\"{\i}ve height. Indeed, this follows because
\[
x_j(n)=Q_j(n,\rho)_{\theta_1,\ldots,\theta_r}\eul^{f_j(n)}
\]
for~$j=1,\ldots,r$, where~$Q_j(n,\rho)_{\theta_1,\ldots,\theta_r}$
as given by~\eqref{eq:6} is non-zero
because~$P_j(X)$ is non-zero by Lemma \ref{lemma:leadcoefficientanddegree}.
Since from now on~$X=n$ is the only variable, we omit the dependence on~$\rho,\theta_1,\ldots,\theta_r$ when we refer to the polynomials~$Q_j(X,\rho)_{\theta_1,\ldots,\theta_r}$.
Indeed, if the degree of~$P_j(X)$ is~$d_j>0$, then the degree of~$Q_j(X,\rho)_{\theta_1,\ldots,\theta_r}$ is~$8d_j-3>0$, otherwise~$P_j(X)$ is a non-zero constant~$P_j(0)$, and
$Q_j(X)$ is the non-zero constant~$P_j(0)$. If~$r=1$, then~$Q_1(n,\rho)_{\theta_1}\eul^{f_1(n)}$ is the only term in the left-hand side of~\eqref{eq:t}, so~\eqref{eq:t} is impossible. Thus,~$r\ge 2$, so one of~$\theta_1$ and~$\theta_2$ is not in~${\mathbb Q}\pi$.
Thus, one of~$\eul^{{\imag} \theta_1}$ or~$\eul^{{\imag}\theta_2}$ is not a root of unity, which implies the inequality~\eqref{eq:height} by considering just the first two coordinates of~${\bf x}$, namely~$x_1(n)$ or~$x_2(n)$. We assume that~$n$ is large, in particular that it is outside the finite set of zeros of all the non-zero polynomials~$Q_1(X),\ldots,Q_L(X)$.
It is then an immediate consequence of
Schmidt's subspace theorem~\cite{MR0314761}
that the solutions~${\bf x}(n)$ of the form~\eqref{eq:vector}
to the inequality
\[
\sum_{i=1}^{L_1} x_i =O\Bigl(H({\bf x})^{-\delta_1/\kappa}\Bigr),
\]
which is implied by~\eqref{eq:t} via~\eqref{eq:height}, live in finitely many
subspaces of~${\overline{\mathbb Q}}^{\#{\mathcal D}}$.
That is, there exist finitely many non-zero
vectors~${\bf d}\in \{{\bf d}^{(1)},\ldots,{\bf d}^{(u)}\}\subset {\overline{\mathbb Q}}^{\#{\mathcal D}}$
with the property that
on writing~${\bf d}^{(\ell)}=(d_{i}^{(\ell)})_{1\le i\le L}$
we must have
that there exists
some~$\ell\in \{1,\ldots,u\}$ such that
\begin{equation}
\label{eq:Sunit}
\sum_{i=1}^{L_1} d_{i}^{(\ell)} x_i=0.
\end{equation}
All this was in the case when the left-hand side of~\eqref{eq:t} is non-zero.
If it is zero, we get the equation~\eqref{eq:Sunit} at once,
with~$d_i^{(\ell)}=1$ for~$1\le i\le L_1$. So it remains to look
at equations of the form~\eqref{eq:Sunit}. For each~$n$ satisfying~\eqref{eq:Sunit} the left-hand side can be non-degenerate or degenerate.
Here, non-degenerate means that the sum over any proper subset of~$\{1,\dots,L_1\}$
on the left-hand side of~\eqref{eq:Sunit}
does not sum to zero.
In any case an equation of the form~\eqref{eq:Sunit}
may be thought of as a sum with a bounded number of terms
of sums over disjoint subsets of~$\{1,\dots,L_1\}$
each of which comprises a non-degenerate equation.
That is, if~\eqref{eq:Sunit} is degenerate, then we may write
\[
\{1\le i\le L_1: d^{(\ell)}_i\ne 0\}=\bigcup_{j=1}^t \Gamma_j^{(\ell)},
\]
where the right-hand side is a partition into~$t\ge 2$
non-empty subsets~$\Gamma_j^{(\ell)}$ of the set on the left, and such that
\[
\sum_{i\in \Gamma_j^{(\ell)}} d_i^{(\ell)} x_i=0
\]
for~$1\le j\le t$,
where
each such equation is non-degenerate, meaning
that no proper subsum on the left is zero.
So we may assume without loss of generality that
equation~\eqref{eq:Sunit} is non-degenerate.
As in the proof of the finiteness of the number of
non-degenerate solutions to~${S}$-unit equation
(see Schlickewei~\cite{MR1069241} for example;
technically,~\eqref{eq:Sunit}
is not an~$S$-unit equation since in addition to the
elements~$\eul^{{\imag} f_i(n)}$ which belong to the multiplicative subgroup
of~${\mathbb C}^*$ generated by~$\{\alpha_1,\ldots,\alpha_r,\rho\}$, the elements~$x_i(n)$
also involve the polynomials~$Q_i(n)$ in~$n$, but their heights are of size~$n^{O(1)}=\eul^{o(n)}$, an amount which is negligible, so the argument goes through),
we are lead to the conclusion that for each
$\ell\in \{1,\ldots,u\}$ there exist~$i_1^{(\ell)}\ne i_2^{(\ell)}$ and a finite set
of complex numbers~${\mathcal D}^{(\ell)}_{i_1^{(\ell)},i_2^{(\ell)}}$ such that for each such~$n$, there is some~$\ell\in \{1,\ldots,u\}$
with
\[
\frac{x_{i_1}^{(\ell)}(n)}{x_{i_2}^{(\ell)}(n)}\in {\mathcal D}^{(\ell)}_{i_1^{(\ell)},j_1^{(\ell)},i_2^{(\ell)},j_2^{(\ell)}}.
\]
We omit the dependence on~$\ell$ for the rest of this argument
Hence,
\begin{equation*}
\frac{Q_{i_1}(n)}{Q_{i_2}(n)}\eul^{{\imag}(f_{j_1}(n)-f_{j_2}(n))}\in {\mathcal D}_{i_1,j_1,i_2,j_2}.
\end{equation*}
We thus get that
\begin{equation}
\label{eq:tt}
\eul^{{\imag}(f_{j_1}(n)-f_{j_2}(n))}\in {\mathcal D}_{i_1,j_1,i_2,j_2} \!\!\left(\frac{Q_{i_2}(n)}{Q_{i_1}(n)}\right)\!\!.
\end{equation}
If~$f_{j_1}(X)\ne f_{j_2}(X)$ then, by previous arguments,
for large~$n$ the number on the left-hand side of~\eqref{eq:tt}
above is not a root of unity
so its height is at least~$\eul^{\kappa_2 n}$ for some positive constant~$\kappa_2$, while the height of the number on the right-hand side of~\eqref{eq:tt} is~$n^{O(1)}$.
Thus,~\eqref{eq:t} cannot hold for large~$n$ unless~$f_{j_1}=f_{j_2}$.
This almost proves step~(a).
To complete the argument,
fix some~$j_1\in \{1,\ldots,L_1\}$
and apply the argument above to derive
an equation like~\eqref{eq:Sunit}. If it is non-degenerate and involves~$j_1$ (so~\eqref{eq:Sunit} holds for infinitely many~$n$ with some~$\ell$ such that~$d_{j_1}^{(\ell)}\ne 0$), we are done.
If not, we pick some~$j$ such that~$d_j^{(\ell)}\ne 0$, express~$x_j(n)$ linearly from~\eqref{eq:Sunit} as
\[
x_j(n)=-\sum_{j'\ne j} \left(d^{(\ell)}_{j'}/d^{(\ell)}_j\right) x_{j'}(n),
\]
and insert this into the left-hand side of~\eqref{eq:4}. Again we get a linear form
in the variables~${x_i(n)}_{1\le i\le L_1, i\ne j}$
(that is, in a smaller number of variables)
which involves~$x_{j_1}(n)$ and which is ``smaller'' to which we
may apply the same argument.
Eventually, after finitely many steps, we get to an equation like~\eqref{eq:Sunit} involving
our chosen~$j_1$ and some other indices with infinitely many non-degenerate solutions
in~$n$, showing that~$f_{j_1}=f_{j_2}$ for some~$j_2\ne j_1$, which proves step~(a).

Step~(b) is immediate.
Indeed, we have~$f_i={\imag}\theta_i (4X^2+4X)$ for~$i=1,\ldots,r$ and~$\theta_i\ne \theta_j$ for~$i\ne j$ in~$\{1,\ldots,r\}$ so~$f_i(X)=f_j(X)$ is impossible with distinct indices~$i,~j$ both in~$\{1,\ldots,r\}$.

For step~(c), for each~$i\in \{1,\ldots,r\}$, let~$j_i>r$ be such that~$f_i=f_{j_i}$. Matching leading coefficients in~$f_i(X)=f_{j_i}(X)$ (as polynomials in~$X$) and dividing across by~$4$, we get
\begin{equation*}
\label{eq:uuu}
\theta_i=\sum_{j\in I_i} (d_{i,j}/4) \theta_j.
\end{equation*}
Here,~$d_{i,j}$ is the leading coefficient of~$m_{i,j}(X)$. As noted above,~$d_{i,j}>0$
and
\[
\sum_{j\in I_i}d_{i,j}=4,
\]
so~$\theta_i$ is in the interior of the convex hull of
the set~$\{\theta_j\mid j\in I_i\}.$
If~$I_i$ has only two elements, one
of which is~$i$ itself, we deduce,
from~$I_i=\{i,j\}$, that~$(4-d_{i,i})\theta_i=d_{i,j} \theta_j$,
which is impossible since~$\alpha_i/\alpha_j$ is not a root of unity.
Thus, either~$I_i$ does not contain~$i$,
or it does contain~$i$ and has at least~$3$ elements.
So,~ each~$\theta_i$ is in the convex hull of the remaining ones (and all these numbers are in the interval~$(-\pi,\pi]$). Picking~$i$
to correspond to the smallest~$\theta_i$, we get a contradiction.

A different way of seeing this last step is to think
of~$(\theta_1,\ldots,\theta_r)$ as a solution~${\mathbf x}$
to the linear system of equations~${\mathbf A}{\mathbf x}={\mathbf x}$,
where~${\mathbf A}$ is the~$r\times r$ matrix
with entry~$d_{i,j}/4$
in position~$(i,j)$ if~$i\in \{1,\ldots,r\}$ and~$j\in I_i$, and~$0$ otherwise.
Then~${\mathbf A}$ is a matrix whose entries are non-negative,
has row sums equal to~$1$ and each row contains at least two non-zero entries.
It is straightforward
to see that the  eigenspace corresponding to the eigenvalue~$1$
of such a matrix
is one dimensional, and is spanned by~$(1,1,\ldots,1)^T$.
Hence,~$\theta_i=\theta_j$ for~$i=1,\ldots,r$, which is a contradiction.

This finishes the proof of step~(c) and the first proof of the main theorem.

\section{A~$p$-adic proof of the main theorem}

In this section we give a different proof of
a slightly stronger statement than Theorem~\ref{thm:1},
by reducing both the elliptic
and the linear recurrence sequence modulo carefully
chosen primes, and finding incompatible behaviours.
Two remarks are in order here. First, we will need
to call on results elsewhere that guarantee a sufficient
supply of primes with the required properties. Second,
the arithmetic argument here is one approach
and there may be others.

\begin{theorem}\label{thm:2}
Let~$E:y^2=x^3+Ax+B$ be an elliptic
curve defined over the rationals,
let~$P=(x_1/z_1^2,y_1/z_1^3)\in E(\Q)$ be a point
of infinite order, let~$\seqn{z_n}$
be an integer sequence satisfying~$nP=(x_n/z_n^2,y_n/z_n^3)$,
and let~$\seqn{u_n}$ be a linear recurrence
sequence. Then there is an infinite
set of primes~$p$ with the property that
a period of~$\seqn{u_{n^2}}$ modulo~$p$
cannot be a period of~$\seqn{z_n}$ modulo~$p$.
In particular, no linear recurrence
sequence~$\seqn{u_n}$ has the property
that~$\seqn{z_{n^{\vphantom{2}}}}=_e\seqn{u_{n^2}}.$
\end{theorem}

This approach permits some arithmetic perturbation of
the sequences without affecting the conclusion. Specifically,
if~$\seqn{z_n}$ is replaced by any sequence~$\seqn{z_nw_n}$
where the set of primes dividing any term of~$\seqn{w_n}$
is finite, then the same conclusion holds. This allows us, in particular, to assume that the sign of $z_n$ is 
arbitrarily chosen. This means once
again that the proof gives the same result for elliptic
divisibility sequences defined in terms of the bi-linear
recurrences~\eqref{edsrecurrence1}
and~\eqref{edsrecurrence2}.

\subsection{Orders of points on elliptic curves}

For a prime~$p$,
let~$E({\mathbb F}_p)$ be the set of solutions modulo~$p$ of the equation~\eqref{eq:1}
reduced modulo~$p$, together with the  point at infinity~$O$.
Write, in the usual notation,
\[
\#E({\mathbb F}_p)=p-a_p+1
\]
for the number of~$\F_p$-points on~$E$.
Then~$a_p\in(-2{\sqrt p},2{\sqrt p})$
by Hasse's theorem, and if~$p\nmid \Delta_E$,
then~$E({\mathbb F}_p)$ forms a group with the group law
inherited from the Mordell--Weil
group law on~$E$ reduced modulo~$p$.
When~$p$ divides~$\Delta_E$, we have~$a_p\in \{0,\pm 1\}$
(we refer to Silverman~\cite[Ch.~5]{MR817210}
for standard results on~$E(\F_p)$).
If~$p\nmid \Delta_E z_1$, then~$P=(x_1/z_1^2,y_1/z_1^3)$ can be
thought of as
a point on~$E({\mathbb F}_p)$
which is not the identity, and from now on we make the
assumption that~$P$ is a non-torsion point on~$E(\Q)$
and that~$p\nmid \Delta_E z_1$.
We also claim that -- for our purposes -- we may safely
assume that~$x_1y_1\ne 0$. To see this, notice first
that~$y_1\ne 0$ (since if~$y_1=0$ then~$P$ is of order~$2$ in~$E(\Q)$).
If~$x_1=0$, then after replacing~$P$ by~$2P$
(which is still of infinite order)
we have~$x_1\ne 0$.
Furthermore, the order of~$2P$ modulo~$p$ either equals the order
of~$P$ modulo~$p$, or equals half of it (depending of whether the order of~$P$ modulo~$p$ is odd or even). Thus for any odd prime~$q$
the order of~$2P$ modulo~$p$ is a multiple of~$q$
if and only if the order of~$P$ modulo~$p$ is a multiple of~$q$.
Hence, for the purpose of deciding whether the order of~$P$ modulo~$p$
is a multiple of~$q$ or not, we may replace, if we wish,~$P$ by~$2P$
and so assume that~$x_1y_1\neq0$ below.

Now let~$q$ be a fixed large prime.
We will ask what can be said about the set of primes~$p$
with the property that the order of~$P\in E({\mathbb F}_p)$
is divisible by~$q$.
In order to do this, we will make use of recent work of
Meleleo and Pappalardi
(which may be found in the thesis of Meleleo~\cite{GMthesis};
there is also a video presentation by Pappalardi~\cite{MPvideo}
of some of the results; the density
statement we need may also be found in a paper
of David and Wu~\cite{MR2932170}).
Before doing this, we need to recall some group-theoretic
properties of~$E(\F_p)$.

Let~$E[q]=\{Q\mid qQ=O\}$ denote the subgroup of~$q$-torsion points in the curve~$E$.
As an~${\mathbb F}_q$-vector space,~$E[q]$ can be identified with~${\mathbb F}_q^2$.
Adjoining the coordinates of the points~$Q\in E[q]$ to~$\Q$
gives a Galois extension of~$\Q$ with Galois group
isomorphic to a subgroup of~${\rm GL}_2({\mathbb F}_q)$ (we
will supress this isomorphism and simply speak of the Galois
groups arising as being given by matrix or affine groups).
Serre's open image theorem~\cite{MR0387283}
says that there exists a positive integer~$\Delta_{1,E}$
depending on~$E$
with the property that if~$q\nmid\Delta_{1,E}$, then this Galois group
is all of~${\rm GL}_2({\mathbb F}_q)$ (We assume that~$\Delta_{1,E}$
is already divisible by all the prime factors of~$\Delta_E$).

Suppose now that we want to study the density of the
set of primes~$p$ such that~$a_p$ and~$p$ have prescribed
values modulo~$q$, say~$a$ and~$b$. Then one can identify
the Frobenius action of such a prime~$p$ with the equivalence class
of a~$2\times 2$ matrix in~${\rm GL}_2({\mathbb F}_q)$ whose
trace is~$a$ modulo~$q$ and whose determinant
is~$b$ modulo~$q$. That is, for given residue
classes~$a$ and~$b\not\equiv 0$ modulo~$q$, the density
\[
\lim_{x\to\infty}
\frac{\#\{p\le x\mid a_p\equiv q\pmod q~{\text{\rm and}}~p\equiv b\pmod q\}}{\pi(x)}=\delta_{q;a,b},
\]
exists and equals
\begin{equation}\label{eq:positiveproportion}
\delta_{q;a,b}=\frac{\#\{J\in {\rm GL}_2({\mathbb F}_q)\mid
{\rm tr}(J)=a,~{\text{\rm and}}~{\rm det}(J)=b\}}{\#{\rm GL}_2({\mathbb F}_q)}
\end{equation}
where we write as usual~$\pi$ for the prime counting function.
In particular,~$\delta_{q;a,b}$ is always positive since the conditions
in~\eqref{eq:positiveproportion} define a positive proportion of~$\rm{GL}_2(\F_q)$.
Assume next that we want to add the point~$P$ to the picture and see what happens
to its order in~$E({\mathbb F}_p)$ modulo~$q$.

We claim that~$E_P[q]=\{R\mid qR=P\}$ is a two-dimensional affine~$\mathbb{F}_q$ vector
space. To see this, pick some~$R_0\in E_P[q]$ and
notice that~$R\in E_P[q]$ if and only if~$R-R_0\in E[q]$,
itself
identified with
a two-dimensional~$\mathbb{F}_q$ vector space.
We also adjoin the
coordinates of the points of~$E_P[q]$ to~$\Q$,
in addition to the coordinates of the points in~$E[q]$.
Then, by an analogue of Serre's open mapping theorem
due to Bachmakov~\cite{MR0269653} (an accessible and
thorough treatment of the results outlined there
may be found in a paper of Ribet~\cite{MR552524}),
there exists a constant~$\Delta_{2,E,P}$
depending both on~$P$ and~$E$ such that if~$q\nmid\Delta_{2,P,E}$,
then the Galois group of this extension is the full
group of affine
transformations of a
two-dimensional affine~${\mathbb F}_q$--space, namely
\[
{\text{\rm Aff}}(E_P[q])={\rm GL}_2(\F_q)\ltimes\F_q^2,
\]
where~${\text{\rm GL}}_2({\mathbb F}_q)$
acts on~${\mathbb F}_q^2$ by linear automorphisms.
That is, the group law is~$(\phi,u)\circ (\psi,v)=(\phi \psi,\phi(v)+u)$.
We assume that~$\Delta_{2,E,P}$ contains all the prime factors of~$\Delta_{1,E}$
and of~$x_1y_1z_1$ (as pointed out above, we may assume that~$x_1y_1\neq0$).
By~\cite{GMthesis}, it follows that
if~$q$ does not divide~$\Delta_{2,E,P}$, then
\[
\lim_{x\to\infty}\negmedspace\frac{\#\{p\negthinspace\le\negthinspace x\mid a_p\negthinspace\equiv\negthinspace q
\negmedspace\negmedspace\pmod p,~p\negthinspace\equiv\negthinspace b\negmedspace\negmedspace\pmod p,~q\mid {\text{\rm ord}}_{E({\mathbb F}_p)}(P)\}}{\pi(x)}
=
\delta_{q;a,b,P},
\]
where
\[
\delta_{q;a,b,P}=\frac{\#\{(J,u)\in {\rm GL}_2({\mathbb F}_q)\mid {\rm tr}(J)=a,~{\rm det}(J)=b,~u\not\in {\text{\rm Im}}(J-I_2)\}}{\#\left({\rm GL}_2({\mathbb F}_q)\ltimes {\mathbb F}_q^2\right)}.
\]
Note first of all that~$a$ and~$b$ have to be chosen
so that~$p-a_p+1=b-a+1$ is a multiple of~$q$,
so in particular~$b\equiv a-1\pmod q$.
Now if
\[
(J,u)=\left(\left(\begin{matrix} a-1 & -1\\ 0 & 1\\ \end{matrix}\right),\left(\begin{matrix} 1\\ 1\\ \end{matrix} \right)\right)
\]
then~$\rm{tr}(J)=a$,~$\det(J)=a-1=b$, and
\[
u\not\in {\text{\rm Im}}(J-I_2)=\left\{\left(\begin{matrix} x\\ 0\end{matrix}\right), x\in {\mathbb F}_q\right\}.
\]
This shows that~$\delta_{q;a,a-1,P}$ is positive.
We record this conclusion as the following theorem.

\begin{theorem}\label{thm:3}
Let~$a\ge 2$ be an integer,~$E$ be an elliptic curve defined over~$\Q$
with a point of infinite order~$P\in E(\Q)$.
Then there exists~$\Delta=\Delta(E,P)$
such that if a fixed prime~$q$ does not divide~$\Delta$,
then the set of primes~$p\equiv a-1\pmod q$ with~$a_p\equiv a\pmod q$ and~$P\mod q$ having order a multiple of~$q$
in~$E(\F_p)$ has density~$\delta_{q;a,a-1,P}>0$.
\end{theorem}

\subsection{Proving Theorem \ref{thm:2}}

Let~$a=3$, let~$q$ be a fixed sufficiently large prime
(large in a sense that will be made more precise later),
and let~$P_{q;3,2,P}$ be the set of primes~$p$
with the property in the statement of Theorem~\ref{thm:2}.
Finally, let~$p$ be a large prime in~$P_{q;3,2,P}$.
In particular, we may assume that~$p$ does not divide
the denominators, nor the norms (from~$\K$ to~$\Q$) of the numerators of any of the
polynomials~$P_i(X)\in\K[X]$ appearing in the formula~\eqref{eq:Binet}
for the terms of the linear recurrence sequence,
and that~$p$ does not divide the last coefficient~$c_k$ of
the characteristic polynomial~$\Psi(X)$ either.
We put
\[
L={\rm{lcm}}\{p^j-1\mid 1\le j\le d\}.
\]
Note that since by construction~$p\equiv 2$ modulo~$q$,
we have~$p^j-1\equiv 2^j-1$ modulo~$q$
for~$j=1,\dots,k$. Thus, for large~$q$, we have~$q\nmid L$.
Let~$T$ be the period modulo~$p$ of~$\seqn{z_n}$.
It follows from a theorem of Silverman~\cite[Th.~1]{MR2178070})
that~$T\mid 2(p-2)\#E({\mathbb F}_p)$.
Further, since the order of~$P$ modulo~$p$ is divisible by~$q$, it follows that
\[
q\mid T\mid 2(p-1)(p-a_p+1).
\]
To get a contradiction, we work with the
sequence~$\seqn{u_{n^2}}$
for large~$n$, and show that its period modulo~$p$ is coprime to~$q$. This will give us the contradiction.

Assume therefore that~\eqref{eqn:eventualequivalence} holds, so in particular
there is some~$n_0$ for which
\begin{equation}\label{eq:cong}
u_{(n+mT)^2}\equiv u_{n^2}\pmod p
\end{equation}
for all~$n\ge n_0$ and~$m\ge0$.
Let~$\pi$ be a prime ideal of~${\mathbb K}$
lying above the rational prime number~$p$.
The congruence~\eqref{eq:cong} together with Binet's formula~\eqref{eq:Binet}
give
\begin{equation}\label{eq:Sums}
\sum_{i=1}^s \alpha_i^{n^2}(P_i((n+mT)^2)\alpha_i^{2mnT+m^2T^2}-P_i(n^2))\equiv 0\pmod {\pi}.
\end{equation}
Write~$\mathcal{S}$ for the set of primes dividing~$T$
together with all primes smaller than some~$p_0$,
where~$p_0$ is a sufficiently large number to be determined later.
Let~$N$ be the largest divisor of~$L$
composed only of primes from~${\mathcal S}$.
Suppose that~$m=pN\ell$ for some integer~$\ell\ge 0$
in~\eqref{eq:Sums} and use the fact that
\[
P_i((n+pN\ell)^2)\equiv P(n^2)\pmod \pi,
\]
to deduce that
\begin{equation}\label{eq:www}
\sum_{i=1}^s \alpha_i^{n^2} P_i(n^2)(\beta_i^{2\ell n +pNT\ell^2}-1)\equiv 0\pmod {\pi}
\end{equation}
where~$\beta_i:=\alpha_i^{pNT}$ for~$1\le i\le s$.
Once again we postpone the lengthy proof of the next lemma
to the appendix.

\begin{lemma}\label{lemma:deducedcongruences}
If~$p_0$ is sufficiently large,
then the congruence~\eqref{eq:www} implies
that
\[
\beta_i\equiv 1\pmod \pi
\]
for~$i=1,\ldots,s$.
\end{lemma}

Thus~$\alpha_i^{pTN}\equiv 1\pmod \pi$
for any prime ideal~$\pi$ of
the ring of integers~${\mathcal O}_{\mathbb K}$
lying above the prime~$p$.
However, the order of~$\alpha_i$ modulo~$\pi$
divides~$L$, and~$L$ is neither a multiple of~$q$,
nor of~$p$. Also, by construction,~$p$ divides
neither~$L$ nor~$p-a_p+1$.
So, writing~$b_q$ for the exponent of~$q$ in~$T$,
we get that~$\alpha_i^{NT/q^{b_q}}\equiv 1\pmod \pi$.
Since~$p$ is large (and in particular,~$p$
does not divide the discriminant of~${\mathbb K}$),
we conclude that~$p$ splits in distinct prime
ideals~$\pi$ in~${\mathcal O}_{\mathbb K}$.
The above argument then shows
that~$\alpha_i^{NT/q^{b_q}}\equiv 1\pmod p$
for all~$i=1,\ldots,r$.
By the Binet formula~\eqref{eq:Binet},
this means that~$pNT/q^{b_q}$ is a period
of~$\seqn{u_{n^2}}$
modulo~$p$.
By the assumption~\eqref{eqn:eventualequivalence},
it follows that~$pNT/q^{b_q}$
is also a period of~$\seqn{z_n}$.
Hence,~$T\mid pNT/q^{b_q}$, which is not
possible by the definition of~$b_q$.
This contradiction completes the~$p$-adic proof.

\section*{Appendix}

We assemble here some of the calculations used earlier.
Write
\[
P(X)=a_0X^d+a_1X^{d-1}+a_2X^{d-3}+\cdots+a_d.
\]
For some non-zero number~$\alpha$ consider the polynomial
\begin{align}
Q(X)
&:=
P((2X+1)^2)\nonumber\\
&\qquad\qquad-\alpha^3\left(P((X+2)^2)P(X^2)^3-
P((X-1)^2)P((X+1)^2)^3\right).\label{eq:6}
\end{align}

\begin{lemma}\label{lemma:leadcoefficientanddegree}
If~$d=0$ then~$Q(X)$ is the constant~$a_0$, and if~$d>0$ then
\[
Q(X)=-4da_0^4\alpha^3 X^{8d-3}+\mbox{monomials of lower order in~$X$}.
\]
\end{lemma}

\begin{proof}[Proof of Lemma~\ref{lemma:leadcoefficientanddegree}]
If~$d=0$, then~$Q(X)=
a_0^{\vphantom{3}}-\alpha^3(a_0^{\vphantom{3}}a_0^3-a_0^{\vphantom{3}}a_0^3)
=a_0^{\vphantom{3}}$ is constant.
Thus we may assume that~$d>0$.
Computer experiments with Mathematica for~$d=1,2,3$
suggest that the degree of
the polynomial
\begin{equation}
\label{eq:5}
P((X+2)^2)P(X^2)^3-P((X-1)^2)P((X+1)^2)^3
\end{equation}
is~$8d-3$, and the leading coefficient is~$4da_0^4$,
motivating the statement of the lemma.
To verify this, we compute the first three coefficients
of~$P((X+i)^2)$ for~$i=-1,0,1,2$, factor~$X^{8d}$ in the expression~\eqref{eq:5},
change variables using the substitution~$y=1/X$
inside the parentheses, and compute the order
of the resulting expression in~$y$. For example,
\begin{align*}
P((X+2)^2)
&=
a_0(X+2)^{2d}+a_1(X+2)^{2d-2}+\cdots\\
&=
a_0X^{2d}+4d a_0 X^{2d-1}+\left(\vphantom{\binom{1}{2}}\right.
\underbrace{4\binom{2d}{2} a_0+a_1}_{\Sigma_1}\left.\vphantom{\binom{1}{2}}\right)X^{2d-2}\\
&\qquad\qquad
+\left(\vphantom{\binom{1}{2}}\right.
\underbrace{8\binom{2d}{3} a_0+2(2d-2)a_1}_{\Sigma_2}\left.\vphantom{\binom{1}{2}}\right)X^{2d-3} +\cdots\\
P(X^2)
&=
a_0X^{2d}+a_1 X^{2d-2}+\cdots\\
P((X-1)^2)
&=
a_0 (X-1)^{2d}+a_1(X-1)^{2d-2}+\cdots\\
&=
a_0 X^{2d}-2d a_0 X^{2d-1}+
\left(\vphantom{\binom{1}{2}}\right.
\underbrace{\binom{2d}{2} a_0+a_1}_{\Sigma_3}\left.\vphantom{\binom{1}{2}}\right)X^{2d-2}\\
&\qquad\qquad
+\left(\vphantom{\binom{1}{2}}\right.
{-\binom{2d}{3} a_0-(2d-2)a_1}\left.\vphantom{\binom{1}{2}}\right)X^{2d-3}+\cdots\\
P((X+1)^2)
&=
a_0 X^{2d}+2d a_0 X^{2d-1}+\left(\vphantom{\binom{1}{2}}\right.
{\binom{2d}{2} a_0+a_1}\left.\vphantom{\binom{1}{2}}\right)X^{2d-2}\\
&\qquad\qquad
+\left(\vphantom{\binom{1}{2}}\right.\underbrace{\binom{2d}{3} a_0+(2d-2)a_1}_{\Sigma_4}
\left.\vphantom{\binom{1}{2}}\right) X^{2d-3}+\cdots
\end{align*}
So, putting~$y=1/X$, it remains to notice that
\begin{align*}
(a_0&+4da_0y+\Sigma_1y^2+\Sigma_2y^3)
\times(a_0+a_1y^2)^3\\
&-(a_0-2da_0y+\Sigma_3y^2-\Sigma_4y^3)
\times
(a_0+2da_0y+\Sigma_3y^2+\Sigma_4y^3)^3\\
&=(4d)a_0^4y^3+\mbox{higher powers of }y,
\end{align*}
as required.
\end{proof}

\begin{proof}[Proof of Lemma~\ref{lemma:deducedcongruences}]
Assume for the time being that this congruence
does not hold.
Up to relabeling the roots~$\alpha_1,\ldots,\alpha_s$,
we may assume that there exist~$s_1<s$ and
indices~$0<i_1<\cdots<i_t=s-s_1$ such that
\[
\beta_1\equiv \cdots \equiv \beta_{s_1}\equiv 1\pmod\pi
\]
and
\begin{equation}\label{eq:congruenceconditionforcontradiction}
\begin{cases}\beta_{s_1+1}\equiv \cdots \equiv \beta_{s_1+i_1}\equiv \gamma_1\pmod \pi \\
\medspace\medspace\vdots \\
\beta_{s_1+i_{t-1}+1}\equiv \cdots \equiv \beta_{s_1+i_t}\equiv \gamma_t\pmod \pi
\end{cases}
\end{equation}
where~$\gamma_i\not\equiv 1\pmod \pi$ for~$i\in \{1,\ldots,t\}$
and~$\gamma_i\not\equiv \gamma_j\pmod\pi$ for
distinct~$i$ and~$j$ in~$\{1,\ldots,t\}$.
Relation~\eqref{eq:www} becomes
\begin{equation}\label{eq:www1}
\sum_{j=1}^t Q_j(n)(\beta_j^{2\ell n+pNT\ell^2}-1)\equiv 0\pmod \pi
\end{equation}
where
\[
Q_j(n)=\sum_{i=s_1+i_{j-1}+1}^{s_1+i_j} \alpha_i^{n^2} P_i(n^2)
\]
for~$j=1,\dots,t$,
with the convention that~$i_0:=0$.
Notice that
\begin{equation}\label{statementforcontradiction}
\mbox{$\beta_1,\dots,\beta_t$
are distinct modulo~$\pi$ and none is congruent to~$1$}
\end{equation}
by~\eqref{eq:congruenceconditionforcontradiction}.
Write
\[
{\textstyle\frac{L}{N}}:=\prod_{r\mid L/N} r^{a_r}
\]
for the prime decomposition of~$\frac{L}{N}$.
For each prime~$r$ dividing~$\frac{L}{N}$,
choose~$n_0$ with the property
\[
\left(\frac{n_0^2+jpNT}{r}\right)
=
1
\]
for all~$j=1,\dots,t$.
To see that
such an~$n_0$ exists, note that for a fixed
prime~$r$, the number of possible residue classes
for such an~$n_0$ is
\[
I_r=\sum_{0\le n\le r-1} \prod_{1\le j\le t} \frac{1}{2} \left(\left(\frac{n^2+jpN}{r}\right)+1\right)+O(1).
\]
The~$O(1)$ term depends on~$t$
and comes from those~$n\in\{0\dots r-1\}$
for which~$n^2+jpN\equiv 0\pmod r$.
To estimate~$I_r$, expand the inner product, separate the main term and change the order of summation for the remainder
to deduce that
\[
2^tI_r=r+\sum_{{\substack{J\subset \{1,\ldots,t\}\\ J\ne\emptyset}}} \sum_{0\le n\le p-1} \!\left(\!\frac{\prod_{j\in J} (n^2+jpN)}{r}\!\right)\!+O(1)=r+O({\sqrt{r}}+1),
\]
where the implied constant in
the~$O$ term depends on~$t$.
For the above estimate, we use Weil's bound with the observation
that if~$r$ does not divide~$pNT$ and is larger than~$t$,
then the polynomial
\[
\prod_{J\subset \{1,\ldots,t\}} (x^2+jpNT)
\]
has only simple roots modulo~$r$.
This shows that~$I_r>0$ for all~$r$ sufficiently large.
So, we choose the prime~$p_0$ such that~$I_r>0$
for all~$r>p_0$. For each such fixed~$r$, fix~$n_0$ modulo~$r$
such that~$n_0^2+jpN$ is a square modulo~$r$
and extend it to~$r^{a_r}$ in some way.
We also choose~$n_0$ modulo~$p$ such
that~$P_i(n_0)\not\equiv 0\pmod p$
for all~$i=1,\ldots,s$. This is certainly possible if~$p>\sum_{i=1}^s {\text{\rm deg}}(P_i(X))$. So far,~$n_0$
has been fixed only modulo~$pL/N$, and we continue to denote by~$n_0$ the smallest possible positive value of such a number in the arithmetic progression of common difference~$pL/N$.

We claim that there are positive integers~$x_{s_1},\ldots,x_{s}$
such that
\begin{equation}
\label{eq:det}
{\text{\rm det}}\left|\begin{matrix} \alpha_{s_1+i_{j-1}+1}^{(n_0+pL/Nx_{s_1+i_{j-1}+1})^2} & \cdots &  \alpha_{s_1+i_j}^{(n_0+pL/Nx_{s_1+i_{j-1}+1})^2}\\
\alpha_{s_1+i_{j-1}+1}^{(n_0+pL/Nx_{s_1+i_{j-1}+2})^2} & \cdots &  \alpha_{s_1+i_j}^{(n_0+pL/Nx_{s_1+i_{j-1}+2})^2}\\
\cdots & \cdots & \cdots \\
\alpha_{s_1+i_{j-1}+1}^{(n_0+pL/Nx_{s_1+i_j})^2} & \cdots &  \alpha_{s_1+i_j}^{(n_0+pL/Nx_{s_1+i_j})^2}\end{matrix}\right|\ne 0
\end{equation}
for~$j=1,\ldots,t$, and will prove this by an induction
argument as follows.
\begin{itemize}
\item The statement is clear
if~$i_j-i_{i_{j-1}}=1$.
\item If~$i_j-{i_{j-1}}=2$ then the
statement holds because the
ratio
\[
\alpha_{s_1+i_j+2}/\alpha_{s_1+i_{j-1}+1}
\]
is not a root of unity by the non-degeneracy of the
linear recurrence sequence.
\item For larger values of~$i_j-i_{j-1}$
the claim follows by induction,
by first choosing~$x_{s_1+i_{j-1}+1},\ldots,x_{s_1+i_j-1}$
with the property
that the minor of size~$(i_j-i_{j-1}-1)\times (i_j-i_{j-1}-1)$
from the upper left corner is non-zero,
expanding the above determinant over the last row
treating~$x_{s_1+i_j}$ as an indeterminate,
and using the fact that the vanishing of the
resulting determinant leads to an~$S$-unit
equation in this last variable which can have
only finitely many solutions~$x_{s_1+i_j}$.
\end{itemize}

Assuming now that~$x_1,\ldots,x_{s-s_1}$ are
fixed positive integers
satisfying~\eqref{eq:det} for all~$j=1,\ldots,t$,
assume that~$p$ is larger than the norm
with respect to the extension~${\mathbb K}\supseteq {\mathbb Q}$
of each of the determinants~\eqref{eq:det}
for~$j=1,\ldots,t$.
Giving~$n$ the values~$n_0+pL/Nx_1,\cdots,n_0+pL/Nx_{s-s_1}$
in turn
and assuming that for
some~$j\in \{1,\ldots,t\}$, we have
that~$Q_j(n_0+pL/Nx_i)\equiv 0\pmod\pi$
for all~$i\in \{s_1+i_{j-1}+1,\ldots,s_1+i_j\}$,
we get the system
\begin{equation}\label{eq:witnesstosolution}
\sum_{i=s_1+i_{j-1}+1}^{s_1+i_j} \alpha_i^{(n_0+pL/Nx_u)^2} P_i(n_0^2)\equiv 0\pmod \pi
\end{equation}
for~$u=s_1+i_{j-1}+1,\ldots,s_1+i_j$.
Write~$\mathbb{F}_q={\mathbb K}[X]/\pi$ for the
residue field.
The relation~\eqref{eq:witnesstosolution}
says that the non-zero vector~$(P_i(n_0^2))_{s_1+i_{j-1}-1\le i\le s_1+i_j}^T$
in~$(\mathbb{F}_q)^{i_j-i_{j-1}}$
is a solution to a homogeneous system of
equations whose determinant~\eqref{eq:det} is non-zero modulo~$\pi$,
which is a contradiction.
It follows that there exists~$n_0$ in the appropriate
residue class modulo~$pL/N$ such that~$Q_j(n_0)$ is non-zero
modulo~$\pi$ for all~$j=1,\ldots,t$.

For each~$j=1,\ldots,t$
and for each~$r$ dividing~$\frac{L}{N}$
we can choose~$\ell_j$ (modulo~$r$) such that
\[
2\ell_j n_0 +pNT\ell_j^2\equiv j\pmod r.
\]
The claimed choice of~$\ell_j$ modulo~$r$
of the above congruences is formally given by
\[
\ell_j\equiv \frac{1}{pNT} (-n_0+{\sqrt{n_0^2+jpNT}})\pmod r,
\]
which exists since
by construction~$r$ does not divide~$pNT$
and~$n_0^2+jpNT$ is a quadratic residue modulo~$r$
and the square root symbol denotes any choice of square root.
By Hensel's lifting
lemma (see~\cite[Sec.~4.3]{MR861410}),
we can extend this
solution~$\ell_j$ defined modulo~$r$
to a solution also written~$\ell_j$
defined modulo~$r^{a_r}$,
and then
by the Chinese Remainder theorem
to a solution again written~$\ell_j$ modulo~$\frac{L}{N}$.
This finally gives a choice of~$\ell_j$ satisfying
\[
2\ell_j n_0+pNT\ell_j^2 \equiv j\pmod {L/N}.
\]
Thus
\[
\beta_u^{2\ell_j n_0+pNT\ell_j^2}= (\alpha_u^{pNT})^{j+\lambda_j L/N}=\alpha_u^{pNTj}\alpha_u^{pTL}
\]
and so
\[
\beta_u^{2\ell_j n_0+pNT\ell_j^2}
\equiv
\alpha_u^{pNTj}
\equiv
\beta_u^{j}\pmod \pi
\]
because~$L$ is a multiple of the order of~$\alpha_u$ modulo~$\pi$, and the above congruences hold for all~$u=1,\ldots,t$.
This means that we can write~\eqref{eq:www1}
as
\[
\sum_{j=1}^t Q_j(n_0) (\beta_j^u-1)\equiv 0\pmod \pi
\]
for all~$u=1,\ldots,t$,
and~${\bf Q}=(Q_j(n_0))_{1\le j\le t}^T$
is not the zero vector in~${\mathbb F}_q^t$.
It follows that
\[
\det\left|\begin{matrix} \beta_1-1 & \beta_2 -1 & \cdots & \beta_t-1\\
\beta_1^2-1 & \beta_2^2-1 & \cdots & \beta_t^2-1\\
\cdots & \cdots & \cdots & \cdots \\
\beta_1^t-1 & \beta_2^t-1 & \cdots & \beta_t^t-1 \\
\end{matrix}
\right|
\]
is divisible by~$\pi$.
Up to a choice of sign, the above determinant is
\[
\prod_{i=1}^t(\beta_i-1) \prod_{1\le i<j\le t} (\beta_i-\beta_j).
\]
Thus either~$\beta_i\equiv 1\pmod \pi$
for some~$i=1,\ldots,t$, or~$\beta_i\equiv \beta_j\pmod \pi$
for some~$1\le i<j\le t$, and neither possibility
is compatible with~\eqref{eq:congruenceconditionforcontradiction}.
This contradiction proves the lemma.
\end{proof}

\section{Acknowledgements}

We are grateful to both Joe Silverman and an anonymous referee for many suggestions that
have improved the exposition, several of which were quite substantial.
One of these -- not implemented here -- is that there may be
an almost entirely geometric argument that gives the main
result, associating the linear recurrence sequence to a
multiplicative group and showing that the relation~\eqref{eqn:eventualequivalence}
then forces the elliptic curve to have impossible reduction
properties. This paper started during a visit of F.~L. to the Max Planck Institute in Fall, 2013 and ended during a conference at CIRM Luminy celebrating the 60th birthday of Professor Igor Shparlinski in April 2016. This author thanks the Max Planck Institute for hospitality and support, Professors Pieter Moree, Francesco Pappalardi and Igor Shparlinski for useful conversations and the organizers of the CIRM event for their invitation.


\providecommand{\bysame}{\leavevmode\hbox to3em{\hrulefill}\thinspace}

\end{document}